\documentclass[pdf]{article}

\usepackage{pdfpages}
\usepackage{hyperref}
\usepackage{xcolor}
\usepackage[utf8]{inputenc}
\usepackage{amsmath}
\usepackage{amsfonts}
\usepackage{amssymb}
\usepackage{amsthm}
\usepackage{marginnote}
\usepackage{enumerate}
\usepackage{graphicx}
\usepackage[normalem]{ulem}
\usepackage{fancyhdr}

\newcommand{\defstyle}[1]{\textbf{#1}}
\newcommand{\myprob}[1]{\mathbb P \left[ #1 \right]}

\newcommand{\omid}[1]{\mathbb E \left[ #1 \right]}

\newcommand{\omidCond}[2]{\mathbb E \left[ #1 \left| #2 \right. \right]}

\newcommand{\norm}[1]{\left| #1 \right|}

\newcommand{\del}[1]{}
\newcommand{\unwritten}[1]{}

\newcommand{\prokhorov}{d_P}

\newcommand{\restrict}[2]{{
		\left.\kern-\nulldelimiterspace 
		#1 
		\vphantom{\big|} 
		\right|_{#2} 
}}

\theoremstyle{theorem}
\newtheorem{theorem}{Theorem}[section]

\theoremstyle{definition}

\newtheorem{example}[theorem]{Example}
\theoremstyle{definition}
\newtheorem{remark}[theorem]{Remark}

\theoremstyle{theorem}

\numberwithin{equation}{section}

\makeatletter
\let\orgdescriptionlabel\descriptionlabel
\renewcommand*{\descriptionlabel}[1]{%
	\let\orglabel\label
	\let\label\@gobble
	\phantomsection
	\edef\@currentlabel{#1}%
	\let\label\orglabel
	\orgdescriptionlabel{#1}%
}

\makeatother

\begin{document}
	
	\title{On the Existence of Balancing Allocations and Factor Point Processes}
	
	\author{Ali Khezeli \footnote{Institute for Research in Fundamental Sciences (IPM), alikhezeli@ipm.ir} , Samuel Mellick \footnote{Jagiellonian University, samuel.mellick@uj.edu.pl}}
	
	\maketitle
	
	\begin{abstract}
		In this article, we show that every stationary random measure on $\mathbb R^d$ that is essentially free (i.e., has no symmetries a.s.) admits a point process as a factor {(i.e., as a measurable and translation-equivariant function of the measure)}. As a result, we improve the results of Last and Thorisson (2022) on the existence of a factor balancing allocation between ergodic pairs of stationary random measures $\Phi$ and $\Psi$ with equal intensities. In particular, we prove that such an allocation exists if $\Phi$ is diffuse and either $(\Phi,\Psi)$ is essentially free or $\Phi$ assigns zero measure to every $(d-1)$-dimensional affine hyperplane. The main result is deduced from an existing result in descriptive set theory, that is, the existence of lacunary sections. We also weaken the assumption of being essentially free to the case where a discrete group of symmetries is allowed.
	\end{abstract}
	
	
	\section{Introduction}
	
	Can we find a perfect matching between two infinite discrete sets in $\mathbb R^d$? The question becomes nontrivial if we add the assumption that the matching is a measurable and translation-equivariant function of the two sets {(we use the term \textit{factor matching} in this case). The answer to this general question is negative, but the more interesting problem is the existence of factor matchings which are \textit{almost surely} perfect, as discussed below.} Given two jointly-stationary point processes $\Phi$ and $\Psi$ (i.e., random discrete sets whose joint distribution is invariant under all translations), for the existence of an almost surely perfect factor matching, it is necessary that $\Phi$ and $\Psi$ have equal \textit{sample intensities} (the latter is just the \textit{intensity} in the \textit{ergodic} case). Conversely, if the sample intensities are equal, stationarity implies that a generalization of the Gale-Shapley stable marriage algorithm results in a perfect matching almost surely~\cite{HoPePeSc09poissonmatching}. This is in fact proved after the continuous analogue of the problem in the novel works~\cite{HoHoPe06} and~\cite{HoPe05}. In these works, a factor \textit{fair tessellation} is constructed for any nonempty stationary point process; i.e., a tessellation of $\mathbb R^d$ into cells of equal volume (in fact, the abstract existence of a factor fair tessellation is implied by the shift-coupling theorem of Thorisson~\cite{Th96}). Since then, various works in the literature have studied generalizations of this problem, have provided various explicit constructions, and have studied properties of the matchings.
	
	In this work, we focus on the abstract existence results in generalizations of the matching problem. We consider \textit{balancing allocations} between measures $\varphi$ and $\psi$ on $\mathbb R^d$, which means maps $T:\mathbb R^d\to\mathbb R^d$ such that $T_*\varphi = \psi$; i.e., $\psi(A)=\varphi(T^{-1}(A))$ for any Borel set $A\subseteq\mathbb R^d$. This generalizes perfect matchings and fair tessellations. More generally, a \textit{balancing transport} is a Markovian kernel on $\mathbb R^d$ that transports $\varphi$ to $\psi$. For stationary random measures $\Phi$ and $\Psi$, one is interested in the existence of factor balancing allocations and transports (which are translation-equivariant and measurable functions of $(\Phi,\Psi)$). 
	See~\cite{ThLa09} and~\cite{HaKh16stabletransport} for the use of factor balancing transport kernels in constructing a \textit{shift-coupling} of a stationary random measure and its Palm version, the existence of which is proved abstractly in~\cite{Th96} (if the sample intensity is constant). See also~\cite{HoPe05} for the use of balancing allocations in constructing \textit{extra head schemes} for the Poisson point process.
	
	In general it is proved that factor balancing transport kernels exist if and only if $\Phi$ and $\Psi$ have equal {sample intensities}~\cite{ThLa09}. If $(\Phi,\Psi)$ is ergodic, this boils down to the equality of the intensities of $\Phi$ and $\Psi$. An explicit construction of an invariant transport is provided in~\cite{HaKh16stabletransport}, which is a generalization of~\cite{HoHoPe06} and is a generalization of the Gale-Shapley stable marriage algorithm. 
	
	The existence of invariant allocations is more complicated. It is convenient to assume that $\Phi$ is diffuse; i.e., has no atoms (otherwise combinatorial complexities appear). In~\cite{LaTo21}, it is proved that if $\Phi$ is diffuse and there exists an auxiliary nonempty point process $P$ as a factor of $\Phi$ and $\Psi$\footnote{In~\cite{LaTo21}, it is assumed that $P$ is a point process on the same probability space as that of $(\Phi,\Psi)$. This depends on the probability space chosen in the model. However, in the canonical probability space $\Omega:=\mathcal M\times \mathcal M$ defined in Subsection~\ref{subsec:definition}, the condition is equivalent to being a factor of $(\Phi,\Psi)$.} (e.g., when $\Psi$ has atoms), then an invariant balancing allocation exists (under the necessary conditions mentioned above). This had also been proved in~\cite{phdthesis} under the extra assumption that $\Phi$ assigns zero measure to every $(d-1)$-rectifiable set. This gives rise naturally to the question of the existence of factor point processes, which is asked in~\cite{phdthesis} (see also~\cite{LaTo21}). In this paper, we answer the problem affirmatively by proving the following theorem:

	\begin{theorem}
		\label{thm:factorPP}
		Let $\Phi$ and $\Psi$ be arbitrary random measures on $\mathbb R^d$. There exists a point process as a translation-invariant factor of $(\Phi,\Psi)$ (resp. of $\Phi$) that is nonempty a.s. if and only if $(\Phi,\Psi)$ (resp. $\Phi$) has no \emph{invariant direction} a.s.
	\end{theorem}
	
	Here, an \textbf{invariant direction} of $(\Phi,\Psi)$ is a vector $t\in\mathbb R^d\setminus\{0\}$ such that $\Phi+\lambda t = \Phi$ and $\Psi+\lambda t = \Psi$ for all $\lambda\in\mathbb R$. Note that if $(\Phi,\Psi)$ is \defstyle{essentially free}; i.e., has no nontrivial translation-symmetry a.s., then there is no invariant direction. Note also that the set of all invariant directions, plus the origin, is a vector space, which is called \textbf{the subspace of invariant directions} of $(\Phi,\Psi)$. 
	
	Using the above theorem, we will prove the following result on the existence of factor allocations, {which improves the results of~\cite{LaTo21}.}
	
	\begin{theorem}
		\label{thm:allocation}
		Let $\Phi$ and $\Psi$ be stationary random measures on $\mathbb R^d$ with equal sample intensities. Assume that $\Phi$ is diffuse a.s. If at least one of the following conditions holds, then there exists an invariant balancing allocation between $\Phi$ and $\Psi$ that is a factor of $(\Phi,\Psi)$:
		\begin{enumerate}[(i)]
			\item \label{thm:allocation-free} $(\Phi,\Psi)$ is essentially free.
			\item \label{thm:allocation-invariant} $(\Phi,\Psi)$ has no invariant direction.
			\item \label{thm:allocation-d-1} $\Phi$ assigns zero measure to every $(d-1)$-dimensional affine subspace.
			\item \label{thm:allocation-translate} $\Phi$ assigns zero measure to every translate of the space of invariant directions of $(\Phi,\Psi)$.
		\end{enumerate}
	\end{theorem}
	
	Theorem~\ref{thm:factorPP} is deduced immediately from a theorem in descriptive set theory, that is, the existence of \textit{lacunary sections}~\cite{Ke92}. This will be discussed in Subsection~\ref{subsec:proofLacunary}. In order to be self-contained and to present the result to probabilists, a direct proof of Theorem~\ref{thm:factorPP} will also be given in Subsection~\ref{subsec:directproof}. This proof is a simplification of the existence of lacunary sections in the special case of random measures on $\mathbb R^d$. 
	
	\begin{remark}
		Theorem~\ref{thm:factorPP} can also be generalized to the more general setting of actions of groups. {The condition of having no invariant direction should be replaced by the condition that the stabilizer of almost every point is a discrete subgroup.} This is closely related to the result of~\cite{AbMe22} showing that every essentially free probability-measure-preserving action of a non-discrete locally compact second-countable group $G$ is isomorphic to a point process of finite intensity on $G$.
	\end{remark}

	\begin{remark}
		A few days before publishing this preprint, an independent preprint~\cite{HuMu23} is published which states the existence of factor balancing allocations under stronger conditions (but does not study the existence of factor point processes). This work assumes that $\Phi$ assigns zero measure to every $(d-1)$-rectifiable set, which is stronger than the assumptions of Theorem~\ref{thm:allocation}.\footnote{It is stated in~\cite{HuMu23} that the condition is sharp, but Theorem~\ref{thm:allocation} shows that the existence can hold under weaker conditions as well. The more precise statement is that the condition cannot be simply removed from the statement of the theorem.} The method of the proof is by using optimal transport and an extension of Monge's theorem to stationary random measures provided in~\cite{Hu16}.
	\end{remark}
	
	\section{The Existence of Factor Point Processes}
	
	{In this section, we prove Theorem~\ref{thm:factorPP}. Two proofs are given, one by using lacunary sections in Subsection~\ref{subsec:proofLacunary} and a direct proof in Subsection~\ref{subsec:directproof}.
		
	}
	
	\subsection{Definitions}
	\label{subsec:definition}
	
	Let $\mathcal M$ be the set of locally-finite Borel measures on $\mathbb R^d$. This space is a Polish space under a modification of the Prokhorov metric, which is denoted by $\prokhorov$ here. A measure $\varphi\in\mathcal M$ is \defstyle{diffuse} if it has no atoms; i.e., $\varphi(\{x\})=0, \forall x\in\mathbb R^d$. For $t\in\mathbb R^d$ and $\varphi\in M$, define $\varphi+t$ by $(\varphi+t)(A):=\varphi(A-t)$ for $A\subseteq \mathbb R^d$. A \defstyle{stationary random measure} is a random element $\Phi$ of $\mathcal M$ such that its distribution is invariant under translations; i.e., $\Phi+t$ has the same distribution as $\Phi$ for all $t\in\mathbb R^d$. Similarly, a pair of random measures $(\Phi,\Psi)$ is (jointly-) stationary if the joint distribution (on $\mathcal M\times \mathcal M$) is invariant under translations. A \defstyle{stationary (simple) point process} is a stationary random measure that is a counting measure almost surely. In other words, it is a stationary random discrete subset of $\mathbb R^d$.

	Given a measure $\varphi\in\mathcal M$, the \textbf{group of translation-symmetries} of $\varphi$ is $H:=H(\varphi):=\{t\in\mathbb R^d: \varphi-t=\varphi\}$. It can be seen that $H$ is a closed subgroup of $\mathbb R^d$. Therefore, it can be decomposed into the sum of a linear subspace $V:=V(\varphi)\subseteq\mathbb R^d$ and a lattice in the orthogonal complement of $V$ (by a lattice we mean the discrete subgroup generated by a basis of the subspace). The subspace $V$ is indeed the \defstyle{subspace of invariant directions} of $\varphi$ defined in the introduction. Similar definitions can be provided for a pair $(\varphi,\psi)$ of measures. A stationary random measure is called \defstyle{essentially free} if $H(\Phi)=\{0\}$ a.s. It has no invariant direction a.s. if $V(\Phi)=\{0\}$ a.s.
	
	A point process $P$ is called a (equivariant) \defstyle{factor} of $\Phi$ if $P$ is equal to a measurable translation-equivariant function of $\Phi$; i.e., $P(\Phi+t) = P(\Phi)+t$ for every $t\in\mathbb R^d$ and for almost every sample of $\Phi$.\footnote{{For this, it is enough that $\forall t: \myprob{P(\Phi+t)=P(\Phi)+t}=1$. See Proposition~B5 of~\cite{bookZi84}.}}

	\subsection{Proof of Theorem~\ref{thm:factorPP} Using Lacunary Sections}
	\label{subsec:proofLacunary}
	
	The following definition and result are borrowed from~\cite{Ke19}. Let $X$ be a {metric} space equipped with a Borel action of a topological group $G$. {In other words, the map $(g,x)\mapsto gx$ is a Borel measurable map from $G\times X$ to $X$.}
	A Borel subset $S\subseteq X$ is called a \defstyle{complete section} if it intersects every orbit of the action. It is called a \defstyle{lacunary section} if there is a neighborhood $U$ of the identity of $G$ such that for every $s\in S$, one has $(U\cdot s) \cap S = \{s\}$. {In other words, in every orbit, one has selected a nonempty set of \textit{uniformly separated} points in a measurable way.
		\begin{example}
			Let $X:=\mathcal N\subseteq\mathcal M$ be the set of discrete subsets of $\mathbb R^d$ and consider the action of $G:=\mathbb R^d$ on $\mathcal N$ by translations. The subset {$\mathcal N_0\cup\{\emptyset\}$, where} $\mathcal N_0=\{\varphi\in\mathcal N: 0\in\varphi\}$, is a complete section, but it is not a lacunary section. However, it would be lacunary in the weaker sense if $U$ would be allowed to depend on $s$. Also, if $U$ is an arbitrary neighborhood of the origin, then $(U +\varphi)\cap \mathcal N_0$ is a finite set. 
			Even in this simple example, it does not seem immediate to find a lacunary section.
		\end{example}
	}
	
	The following is a special case of Theorem~3.10 of~\cite{Ke19}, which is given originally in~\cite{Ke92}:
	
	\begin{theorem}[\cite{Ke92}]
		\label{thm:Kechris}
		If $G$ is a locally compact Polish group acting in a Borel way on a Polish space, then the action admits a complete lacunary section.
	\end{theorem}
	
	The existence of complete lacunary sections in the context of free quasi measure preserving actions goes back to Forrest~\cite{Fo74}. 
	
	We will deduce Theorem~\ref{thm:factorPP} quickly from the above theorem. {A direct proof will also be given in Subsection~\ref{subsec:directproof}, which is a simplification of the proof of~\cite{Fo74}.}

	\begin{proof}[Proof of Theorem~\ref{thm:factorPP}]
		%
		Let $H$ be the group of translation-symmetries of $(\Phi,\Psi)$ and $V$ be the subspace of invariant directions. Observe that if $V\neq\{0\}$ with positive probability, then there exists no nonempty factor point process (otherwise, the point process should also have invariant directions, which is impossible). This proves the \textit{if} side of the claim. 
		
		For the other side, we will use Theorem~\ref{thm:Kechris}.	
		Consider the action of $G:=\mathbb R^d$ on $\mathcal M\times\mathcal M$ by translations. By Theorem~\ref{thm:Kechris} above, there exists a complete lacunary section $S\subseteq\mathcal M\times \mathcal M$ for this action. 
		Define 
		\[P:=P(\Phi,\Psi):=\{t\in\mathbb R^d: (\Phi-t,\Psi-t)\in S\}.\] 
		Since $S$ is a complete section, $P$ is nonempty. In general, $P$ need not be a point process since it contains translated copies of $H$. However, if $H$ is a discrete subgroup of $\mathbb R^d$, it is straightforward to show that $P$ is also a discrete set (since $S$ is lacunary). {Indeed, if $U\subseteq\mathbb R^d$ is the neighborhood of 0 in the definition of the lacunary section $S$, then for every $t\in P$, one observes that $(U+t)\cap P = (U\cap H)+t$, which implies that $P$ is discrete. By the definition of $V$ is Subsection~\ref{subsec:definition}, $H$ is discrete if and only if $V$ is trivial;} i.e., $(\Phi,\Psi)$ has no invariant direction. In this case, $P$ is a translation-invariant factor point process of $\Phi$. So the claim is proved.
	\end{proof}

	\subsection{Direct Proof of Theorem~\ref{thm:factorPP}}
	\label{subsec:directproof}
	
	In this section, we provide a direct proof of Theorem~\ref{thm:factorPP} in order to be self-contained and to show the idea more clearly. In the case of essential freeness, this proof is a simplification of that of~\cite{Fo74}. The latter is more general and considers the continuous actions of locally-compact Polish groups. The proof is much simpler in the setting of the present paper. Also, the arguments are modified to cover the case where $(\Phi,\Psi)$ is not essentially free.

	\begin{proof}[Direct Proof of Theorem~\ref{thm:factorPP}]
		The necessity of the condition is trivial and is shown in Subsection~\ref{subsec:proofLacunary}. So we prove sufficiency here.
		Also, for simplicity of notation, we prove Theorem~\ref{thm:factorPP} for factors of $\Phi$ only, {assuming that $\Phi$ has no invariant direction}. The proof for factors of $(\Phi,\Psi)$ is identical.

		Let us start by an outline of the strategy of the proof. We first produce a factor of $\Phi$ that is an invariant random open set of $\mathbb R^d$ such that its connected components are bounded and it is non-empty on an event $E$ with positive probability. It is then simple to refine this factor open set to a point process. Then, a sort of \textit{measurable Zorn's lemma} argument allows one to enlarge $E$ to a maximal  event (up to null sets) and show that the maximal event has probability one. This shows that $\Phi$ must admit point process factors which are almost surely nonempty.

		Let $H$ be the group of translation-symmetries of $\Phi$ and $V$ be the subspace of invariant directions. Assume $V=\{0\}$ a.s. This implies that $H$ is a discrete lattice in $\mathbb R^d$. Hence, there exists a smallest natural number $N=N(\Phi)$ such that $\bar B_{2/N}(0)\cap H=\{0\}$, where $\bar B_r(0)$ is the closed ball of radius $r$ centered at 0. 
		Consider the shell $K:=K(\Phi):=\{t\in\mathbb R^d: \frac 1 N\leq \norm{t} \leq \frac 2 N\}$. Given a measure $\varphi\in\mathcal M$, define $\theta_K \varphi:=\{\varphi + t: t\in K(\varphi) \}\subseteq \mathcal M$. The definition of $N$ gives that $\Phi\not\in \theta_K\Phi$ a.s. Since $\theta_K\Phi$ is a closed subset of $\mathcal M$, one obtains $\prokhorov(\Phi,\theta_K\Phi)>0$ a.s. Hence, by choosing $\epsilon>0$ sufficiently small, one can assume $\myprob{\Phi\in A_{\epsilon}}>0$, where
		\[
		A_{\epsilon}:=\{\varphi\in\mathcal M: \prokhorov(\varphi, \theta_K\varphi)>\epsilon \}.
		\]
		Observe that $A_{\epsilon}$ is an open subset of $\mathcal M$. We may therefore fix some open subset $B\subseteq A_{\epsilon}$ of diameter less than $\epsilon$ such that $\myprob{\Phi\in B}>0$ (simply express $A_{\epsilon}$ as a countable union of balls of radius less than $\epsilon$, and then pick one that has a positive probability). Define
		\[
		U:=U(\Phi):=\{t\in \mathbb R^d: \Phi-t\in B \}.
		\]
		Observe that $U$ is an open subset of $\mathbb R^d$ which is an equivariant factor of $\Phi$; i.e., $U(\varphi+s) = U(\varphi) + s, \forall \varphi\in\mathcal M, \forall s\in\mathbb R^d$. It is also nonempty with positive probability since $0\in U$ if $\Phi\in B$.
		
		The key property of $U$ is the following: If $t,s\in U$, then $\Phi-t$ and $\Phi-s$ belong to $B$, and hence, $\prokhorov(\Phi-t,\Phi-s)<\epsilon$ (since $B$ has diameter less than $\epsilon$). Therefore, $t-s\not\in K$ by the definition of $A_{\epsilon}$. That is, either $\norm{t-s}< \frac 1 N$ or $\norm{t-s} > \frac 2 N$. As a result, the following defines an equivalence relation on $U$: $t\sim s$ iff $\norm{t-s}<\frac 1 N$. Also, every equivalence class has diameter at most $\frac 1 N$, and hence, is bounded.
		
		Now, one may produce a factor point process $P$ in many ways\footnote{This step is specialized to $\mathbb R^d$ and its proof is more involved for the actions of other groups.}. For instance, for each of the (countably many) equivalence classes $C$ of $U$, choose the least element of the closure $\overline{C}$ according to the lexicographic order. Note that this process is an invariant factor of $\Phi$ and is uniformly separated.
		
		Observe that $P$ is nonempty if and only if $\Phi\in \theta_{\mathbb R^d}B$. Note that $\theta_{\mathbb R^d} B$ is invariant, open and has positive probability. If $\Phi$ were ergodic, then this would be an almost sure event and we would be done. In the general case, let $p\leq 1$ denote the supremum of $\myprob{\Phi\in E}$, where $E\subseteq \mathcal M$ ranges over all open and $\mathbb R^d$-invariant subsets such that there exists a factor point process which is nonempty on $E$. It can be seen that the supremum is attained (if $\myprob{\Phi\in E_n}\to p$, let $P_n$ be a factor point process which is nonempty on $E_n$ and define $P:=P_n$ if $\Phi\in E_n\setminus (E_1\cup \ldots \cup E_{n-1}$)). Also, if $E$ is any such event such that $\myprob{\Phi\in E}<1$, then one can enlarge $E$ a little bit (it is enough to repeat the proof by replacing $\mathcal M$ with $\mathcal M\setminus E$). These two facts imply that the supremum is attained and $p=1$, which implies the claim of the theorem.
		
		An alternative constructive proof of the last step is as follows. First, choose $\epsilon$ as an invariant function of $\Phi$ from the beginning; e.g., the largest number of the form $\epsilon=1/M$ such that $\Phi\in \theta_{\mathbb R^d} A_{\epsilon}$. Then, cover $A_{\epsilon}$ by open sets $B_1,B_2,\ldots\subseteq A_{\epsilon}$ such that each $B_i$ has diameter less than $\epsilon$. Finally, choose the smallest $i$ such that $\Phi\in \theta_{\mathbb R^d} B_i$ and construct $U$ and $P$ using $B_i$ similarly to the above construction. This way, $P$ is nonempty a.s. and the theorem is proved.
	\end{proof}
	
	\section{The Existence of Balancing Allocations}
	\label{sec:allocation}
	
	In this section, we prove Theorem~\ref{thm:allocation} using Theorem~\ref{thm:factorPP}. We start by formalizing the definitions given in the introduction.
	
	Let $\Phi$ and $\Psi$ be stationary random measures on $\mathbb R^d$. 
	Let $I$ be the $\sigma$-field of invariant events in $\mathcal M$; i.e., those events that are invariant under all translations. $\Phi$ is called \defstyle{ergodic} if every event in $I$ has probability zero or one. Ergodicity of $(\Phi,\Psi)$ is defined similarly. 
	The \defstyle{intensity} of $\Phi$ is defined by $\omid{\Phi(C)}$, where $C\subseteq \mathbb R^d$ is an arbitrary Borel set with unit volume. This quantifies the \textit{mean measure per unit volume} of $\Phi$. The \defstyle{sample intensity} of $\Phi$ is the random variable $\omidCond{\Phi(C)}{I}$. Note that if $\Phi$ is ergodic, then its sample intensity is equal to its intensity a.s. {In general, the sample intensity is equal a.s. to the asymptotic density of points per unit volume; i.e., the limit of the number of points of $\Phi\cap [-1,1]^n$ divided by $2^n$ as $n\to \infty$.}
	
	An \defstyle{(equivariant) factor balancing allocation} between $\Phi$ and $\Psi$ is a map $T:\mathbb R^d\to\mathbb R^d$ that depends on $(\Phi,\Psi)$ in a translation-equivariant and measurable way and satisfies $T_*\Phi=\Psi$ a.s. {More precisely, being a factor means $T_{(\Phi+t,\Psi+t)}(x+t)=T_{(\Phi,\Psi)}(x)+t, \forall t\in\mathbb R^d$ and the function $(x,\Phi,\Psi)\mapsto T(x)$ (defined on $\mathbb R^d\times \mathcal M\times\mathcal M$) is measurable. }

	\begin{proof}[Proof of Theorem~\ref{thm:allocation}]
		\eqref{thm:allocation-free}. If $(\Phi,\Psi)$ is essentially free, then it has no invariant direction a.s. So the claim is implied by part~\eqref{thm:allocation-invariant} proved below.
		
		\eqref{thm:allocation-invariant}. If there is no invariant direction a.s., then Theorem~\ref{thm:factorPP} implies that there exists a point process that is nonempty a.s. and is a translation-equivariant factor of $(\Phi,\Psi)$. Using this as an auxiliary point process, Theorem~1.1 of~\cite{LaTo21} constructs a balancing factor allocation.
		
		\eqref{thm:allocation-d-1}. The claim is implied by part~\eqref{thm:allocation-translate}, which will be proved below. The only remaining case is when the subspace of invariant directions is the whole $\mathbb R^d$. In this case, both $\Phi$ and $\Psi$ are multiples of the Lebesgue measure and the claim is trivial.
		
		\eqref{thm:allocation-translate}. Let $V$ be the subspace of invariant directions. First, assume that $(\Phi,\Psi)$ is ergodic\footnote{In the ergodic case, \eqref{thm:allocation-translate} is deduced from part~\eqref{thm:allocation-invariant} in~\cite{phdthesis}. For being self-contained, we include the proof here.}. In this case, $V$ is almost surely equal to a deterministic subspace (since it is a translation-invariant function of $(\Phi,\Psi)$). Let $W$ be the orthogonal complement of $V$. The measures $\Phi$ and $\Psi$ induce two measures on $W$. More precisely, given an arbitrary Borel set $C$ in $V$ with unit volume, for $A\subseteq W$, define $\Phi'(A):=\Phi(A+ C)$ and $\Psi'(A):=\Psi(A+ C)$, where $A+ C:=\{x+y: x\in A, y\in C\}$ is the Minkowski sum of $A$ and $C$. Now, $\Phi'$ and $\Psi'$ are ergodic stationary random measures on $W$ and their intensities are equal to those of $\Phi$ and $\Psi$ (and hence, are equal). Also, $(\Phi',\Psi')$ has no invariant direction. In addition, the assumption of~\eqref{thm:allocation-translate} implies that $\Phi'$ is diffuse. Therefore, part~\eqref{thm:allocation-invariant} implies that there exists a balancing allocation between $\Phi'$ and $\Psi'$ which is a factor of $(\Phi', \Psi')$. If $\tau'$ denotes this allocation (note that $\tau':W\to W$), define the allocation $\tau$ on $\mathbb R^d$ by $\tau(v+w):=v+\tau'(w), \forall v\in V,\forall w\in W$. Then, $\tau$ is a factor allocation which balances between $\Phi$ and $\Psi$ as desired. 
		
		In the general case where $(\Phi,\Psi)$ might be non-ergodic, the spaces $V$ and $W$ might be random and some care is needed to choose the allocation as a measurable factor of $(\Phi,\Psi)$ (a naive ergodic decomposition is not sufficient). In this case, construct $V, W, \Phi'$ and $\Psi'$ as above. Then, let $k:=\mathrm{dim}(W)$ and choose an orthonormal basis $(e_1,\ldots,e_k)$ for $W$ as a measurable function of $W$ (e.g., consider the orthogonal projection of the standard unit vectors of $\mathbb R^d$ on $W$ and use the Gram–Schmidt algorithm). This defines a linear map $L:W\to\mathbb R^k$. Let $\Phi'':=L_*\Phi'$ and $\Psi'':=L_*\Psi'$. By part~\eqref{thm:allocation-invariant}, construct a factor balancing allocation $\tau''$ between $\Phi''$ and $\Psi''$. Then, define the allocation $\tau'$ between $\Phi'$ and $\Psi'$ by $\tau':=L^{-1}\circ \tau''\circ L$, and finally, construct $\tau$ similarly to the previous paragraph. In this construction, every constructed item is a Borel measurable function of the previously constructed items. This implies that $\tau$ is a measurable factor of $(\Phi,\Psi)$ and the claim is proved.
	\end{proof}
	
		\begin{remark}[Measurability]
			As mentioned in the introduction, for the existence of a factor balancing allocation (or transport kernel), it is necessary that $\Phi$ and $\Psi$ have equal sample intensities. This would no longer be necessary if the measurability assumption was removed (one can use the axiom of choice to choose one point in each orbit of $\mathcal M\times\mathcal M$ and use it to construct an equivariant balancing allocation if $\Phi$ is diffuse). Measurability enables one to use the \textit{mass transport principle} and prove the necessity of the equality of the intensities.
		\end{remark}

		\begin{remark}[Remaining Cases]
			The only remaining cases in studying the existence of factor balancing allocations are:
			\begin{enumerate}[(i)]
				\item When $\Phi$ has atoms,
				\item When after some random rotation, $(\Phi,\Psi)$ is of the form $(\Phi'\otimes\mathrm{Leb}, \Psi'\otimes\mathrm{Leb})$, where $\Phi'$ and $\Psi'$ are measures on $\mathbb R^k$, $\mathrm{Leb}$ is the Lebesgue measure on $\mathbb R^{d-k}$ and $\Phi'$ has atoms.
			\end{enumerate}
			By the method of the proof of Theorem~\ref{thm:allocation}, the second case can be reduced to the first case. So it remains to study the existence of factor allocations when $\Phi$ is not diffuse. 
			In this case, combinatorial obstacles appear since the mass of an atom cannot be splitted by an allocation.
		\end{remark}

	\section*{Acknowledgments}
	This work has been completed while the first author was affiliated with INRIA Paris. This work was supported by the ERC NEMO grant, under the European Union's Horizon 2020 research and innovation programme, grant agreement number 788851 to INRIA. We also thank Mir-Omid Haji-Mirsadeghi for suggesting the alternate proof of the last step of the proof of Theorem~\ref{thm:factorPP}.
	
	\bibliography{bib} 

\begin{thebibliography}{10}

\bibitem{AbMe22}
M.~Ab{\'e}rt and S.~Mellick.
\newblock Point processes, cost, and the growth of rank in locally compact
  groups.
\newblock {\em Israel Journal of Mathematics}, 251(1):48--155, 2022.

\bibitem{Fo74}
P.~Forrest.
\newblock On the virtual groups defined by ergodic actions of {$R^{n}$} and
  {${\bf Z}^{n}$}.
\newblock {\em Advances in Math.}, 14:271--308, 1974.

\bibitem{HaKh16stabletransport}
M.~O. Haji-Mirsadeghi and A.~Khezeli.
\newblock Stable transports between stationary random measures.
\newblock {\em Electron. J. Probab.}, 21:Paper No. 51, 25, 2016.

\bibitem{HoHoPe06}
C.~Hoffman, A.~E. Holroyd, and Y.~Peres.
\newblock A stable marriage of {P}oisson and {L}ebesgue.
\newblock {\em Ann. Probab.}, 34(4):1241--1272, 2006.

\bibitem{HoPePeSc09poissonmatching}
A.~E. Holroyd, R.~Pemantle, Y.~Peres, and O.~Schramm.
\newblock Poisson matching.
\newblock {\em Ann. Inst. Henri Poincar\'{e} Probab. Stat.}, 45(1):266--287,
  2009.

\bibitem{HoPe05}
A.~E. Holroyd and Y.~Peres.
\newblock Extra heads and invariant allocations.
\newblock {\em Ann. Probab.}, 33(1):31--52, 2005.

\bibitem{Hu16}
M.~Huesmann.
\newblock Optimal transport between random measures.
\newblock {\em Ann. Inst. Henri Poincar\'{e} Probab. Stat.}, 52(1):196--232,
  2016.

\bibitem{HuMu23}
M.~Huesmann and B.~M{\"u}ller.
\newblock Transportation of random measures not charging small sets.
\newblock {\em arXiv preprint arXiv:2303.00504}, 2023.

\bibitem{Ke92}
A.~S. Kechris.
\newblock Countable sections for locally compact group actions.
\newblock {\em Ergodic Theory Dynam. Systems}, 12(2):283--295, 1992.

\bibitem{Ke19}
Alexander~S Kechris.
\newblock The theory of countable borel equivalence relations.
\newblock {\em preprint}, 2019.

\bibitem{phdthesis}
A.~Khezeli.
\newblock {\em Mass Transport Between Stationary Random Measures}.
\newblock PhD thesis, Sharif University of Technology, 2016.
\newblock (In Persian).

\bibitem{ThLa09}
G.~Last and H.~Thorisson.
\newblock Invariant transports of stationary random measures and
  mass-stationarity.
\newblock {\em Ann. Probab.}, 37(2):790--813, 2009.

\bibitem{LaTo21}
G.~Last and H.~Thorisson.
\newblock Transportation of diffuse random measures on {$\mathbb R^d$}.
\newblock {\em arXiv preprint arXiv:2112.13053}, 2021.

\bibitem{Th96}
H.~Thorisson.
\newblock Transforming random elements and shifting random fields.
\newblock {\em Ann. Probab.}, 24(4):2057--2064, 1996.

\bibitem{bookZi84}
Robert~J. Zimmer.
\newblock {\em Ergodic theory and semisimple groups}, volume~81 of {\em
  Monographs in Mathematics}.
\newblock Birkh\"{a}user Verlag, Basel, 1984.

\end{thebibliography}
	\bibliographystyle{plain}
	
\end{document}